\documentclass[11pt,a4paper]{amsart}         
\usepackage[margin=3cm]{geometry}

\usepackage{amsmath,amsfonts,amssymb,amsthm}
\usepackage{mathrsfs,epsfig,epstopdf,url}
\usepackage{enumerate}
\usepackage{xcolor}
\usepackage{tikz-cd}
\usepackage[all]{xy}
\usepackage{float}
\usepackage{tikz}
\restylefloat{table}
\usepackage{multirow}
\usepackage{array}

\usepackage{orcidlink}
\usepackage{cleveref}

\theoremstyle{definition}
\newtheorem{definition}{Definition}[section]
\newtheorem{example}[definition]{Example}
\theoremstyle{plain}
\newtheorem{theorem}{Theorem}
\newtheorem{lemma}[definition]{Lemma}
\newtheorem{proposition}[definition]{Proposition}

\newtheorem{corollary}[definition]{Corollary}
\newtheorem{remark}[definition]{Remark}
 
\newcommand{\App}{\mathcal{A}^{++}}
\newcommand{\sn}[1]{\|#1\|_s}
\newcommand{\bigsn}[1]{\left\|#1\right\|_s}
\DeclareMathOperator{\spec}{spec}
\DeclareMathOperator{\rad}{rad}

\newcommand{\A}{\ensuremath{\mathcal{A}}}
\newcommand{\F}{\mathbb{F}}

\newcommand{\noproof}{\hfil\qed}

\usepackage[%
backend=bibtex,%
style=numeric-comp,%
sorting=nyt,%
hyperref,%
]{biblatex}

\AtBeginBibliography{}

\renewbibmacro{in:}{%
\ifentrytype{article}
{}
{\bibstring{in}%
\printunit{\intitlepunct}}}

\title[Molnár's Theorem for split $C^*$-by-nilpotent algebras]{Molnár's Theorem for split $C^*$-by-nilpotent algebras}

\author[M.~Galici]{Mario Galici~\orcidlink{0000-0003-2721-4095}}
\author[G.~La Rosa]{Gianmarco La Rosa~\orcidlink{0000-0003-1047-5993}}
\author[M.~Mancini]{Manuel Mancini~\orcidlink{0000-0003-2142-6193}}
\author[G.~P.~Nagy]{Gábor P.~Nagy~\orcidlink{0000-0002-9558-4197}}

\email{mario.galici@unisalento.it}
\email{gianmarco.larosa@unipa.it}
\email{manuel.mancini@uclouvain.be; manuel.mancini@unipa.it}
\email{nagyg@math.u-szeged.hu}

\address[M.~Galici]{Dipartimento di Matematica e Fisica “Ennio de Giorgi”, Università del Salento, via per Arnesano, 73100, Lecce, Italy}

\address[G.~La Rosa, M.~Mancini]{Dipartimento di Matematica e Informatica, Università degli Studi di Palermo, via Archirafi 34, 90123 Palermo, Italy}

\address[M.~Mancini]{Institut de Recherche en Mathématique et Physique, Université catholique de Louvain, chemin du cyclotron 2 bte L7.01.02, B--1348 Louvain-la-Neuve, Belgium.}

\address[G.~P.~Nagy]{Bolyai Institute, University of Szeged, Aradi vértanúk tere 1, H-6720 Szeged, Hungary.}


\begin{document}

\begin{abstract}
Kubo-Ando means are a rich class of important binary operations on the set of positive elements of an operator algebra. The aim of this article is to understand the algebraic structure of these operations. We generalise a recent result of L.~Molnár for $C^*$-algebras to the class of \emph{split $C^*$-by-nilpotent algebras} by using the method of analytic local expansion of binary operations. Eventually, we prove that in any $C^*$-algebra $\A$, the identity $[x,[y,z]]=0$ implies commutativity.
\end{abstract}

\maketitle

\section{Introduction}

Let $\mathcal{A}$ be a unital $C^*$-algebra and let $\mathcal{A}^+$ and $\mathcal{A}^{++}$ denote respectively the set of non-negative and positive elements of $\mathcal{A}$. It is well known that $\mathcal{A}^{++}$ is closed under addition, inverse, and taking square root: for any $x\in \mathcal{A}^{++}$ there is a unique $y\in \mathcal{A}^{++}$ such that $x=y^2$. This implies that the binary operations
\begin{align}
x\circ y&=\frac{1}{\sqrt{x}}\left(x+\sqrt{\sqrt{x}y\sqrt{x}}\right)^2\frac{1}{\sqrt{x}} \label{eq:circdef} \\
x \# y &= \sqrt{x} \left(\frac{1}{\sqrt{x}}y\frac{1}{\sqrt{x}}\right)^\frac{1}{2} \sqrt{x} \label{eq:sharpdef}\\
x * y &= \sqrt{x}y\sqrt{x} \label{eq:stardef}
\end{align}
are well-defined on $\App$. For positive elements $x,y$, $x\# y$ is called the \emph{Kubo-Ando geometric mean,} and $\frac{1}{4}(x\circ y)$ is the \emph{Bures-Wasserstein mean}. These are examples from a larger class of means that play an important role in the theory of $C^*$-algebras, but also in quantum information theory and the theory of optimal transport (see~\cite{Bhatia_2018,MR3958140,MR3992484,MR4091410,molnar2021} and the references therein).

At first sight, the algebraic properties of these operations are far from being obvious. The commutativity of \eqref{eq:sharpdef} is easy, when you know that $a\# b$ is a solution of the equation 
\begin{align} \label{eq:riccati}
xa^{-1}x=b. 
\end{align}
Indeed, \eqref{eq:sharpdef} and \eqref{eq:stardef} are closely related to the binary operation
\begin{align} \label{eq:core}
x \boxplus y = xy^{-1}x,
\end{align}
which can be defined in any abstract group. In fact, \eqref{eq:core} is useful even in algebraic structures with only a weak form of associativity, like octonions, Moufang and Bol quasigroups, etc. (see~\cite{MR0093552} for the definition of the \emph{core} of a Moufang quasigroup). The unique solvability of \eqref{eq:riccati} for positive operators is referred to as the Anderson-Trapp theorem. However, the reference~\cite{AndersonTrapp} is very hard to find, and this is more of a mathematical folklore result. The operations \eqref{eq:stardef} and \eqref{eq:core} are also fundamental in the theory of finite groups; they led to the celebrated $Z^*$-theorem by G.~Glauberman. In today's terminology (by M.~Aschbacher), the link between positive operators and abstract finite groups is given by the concept of \emph{twisted subgroups.} These are subsets $S$ of a group $G$ such that $1\in S$ and $ab^{-1}a\in S$ for all $a,b\in S$. In particular, $a^2\in S$ for $a\in S$. For instance, the set $\mathcal{A}^{++}$ is a twisted subgroup of the group of invertible elements of a $C^*$-algebra $\mathcal{A}$. Another important example is the set $\{g\in G \mid \alpha(g)=g^{-1}\}$ of anti-fixed elements of a group automorphism $\alpha$ of order $2$. 

We claim that \eqref{eq:riccati} has a unique solution in $S$ if and only if $S$ is \emph{$2$-divisible}, i.e., the map $a\mapsto a^2$ is a bijection $S \to S$. This result is also folklore, see~\cite[Theorem 6.14]{MR1899153} for historical details. The necessity of the condition is obvious. For the sufficiency, let $a^{\frac{1}{2}}$ denote the unique element of $S$ such that $(a^{\frac{1}{2}})^2=a$. Then, the operation $a\# b$ is well-defined in $S$ by \eqref{eq:sharpdef}, and $a\# b$ is the unique solution of $xa^{-1}x=b$ in $S$. Indeed, it is straightforward to see that $a\# b$ is a solution. For the uniqueness, assume $xa^{-1}x=ya^{-1}y$. This implies
\[
(a^{-\frac{1}{2}}xa^{-\frac{1}{2}})^2=a^{-\frac{1}{2}}xa^{-1}xa^{-\frac{1}{2}} =a^{-\frac{1}{2}}ya^{-1}ya^{-\frac{1}{2}} = (a^{-\frac{1}{2}}ya^{-\frac{1}{2}})^2,
\]
and $a^{-\frac{1}{2}}xa^{-\frac{1}{2}}=a^{-\frac{1}{2}}ya^{-\frac{1}{2}}$. Hence, $x=y$. 

By a recent result of L.~Molnár~\cite{molnar2021}, given a $C^*$-algebra $\mathcal{A}$, the operation \eqref{eq:circdef} is commutative on $\mathcal{A}^{++}$, and the groupoid $(\mathcal{A}^{++},\circ)$ is a semigroup if and only if $\mathcal{A}$ is commutative. L.~Molnár also gave a complete description of the maps that preserve \eqref{eq:circdef}. On one hand, this sheds light on the algebraic properties of \eqref{eq:circdef}. On the other hand, L.~Molnár's proofs make heavy use of the toolbox of $C^*$-algebras, particularly the Löwner order on the set of self-adjoint elements. 

The purpose of the present paper is to extend L.~Molnár's commutativity and associativity results~\cite[Proposition~2]{molnar2021} to the class of split \emph{$C^*$-by-nilpotent algebras} (see \Cref{def_split}). From now on, we denote by $[x,y]=xy-yx$ the usual Lie bracket.

\begin{theorem} \label{thm:main}
Let $B=\mathcal{A}\oplus N$ be a split $C^*$-by-nilpotent algebra. The operation \eqref{eq:circdef} is well-defined and commutative on $B^{++}$. Moreover, $(B^{++},\circ)$ is a semigroup if and only if $[[x,y],z]=0$ holds for all $x,y,z\in B$. 
\end{theorem}

Concerning the first statement of the Theorem, the authors of~\cite{MR3958140} noticed that for positive definite matrices $A,B$, $AB$ has only positive eigenvalues and, therefore, it has a unique square root $(AB)^\frac{1}{2}$ with positive eigenvalues. Since this is also true for the matrices~$A$ and $A^{-1}\# B$, we have $(AB)^\frac{1}{2}=A(A^{-1}\# B)$ and
\[
A\circ B=A+B+(AB)^\frac{1}{2}+(BA)^\frac{1}{2}.
\]
Thus, $A\circ B$ is commutative for positive definite matrices $A,B$.

The commutativity of \eqref{eq:circdef} is shown for $C^*$-algebras in~\cite{molnar2021} by more elaborate algebraic arguments, using only the fact that, in any $C^*$-algebra, $X=A^{-1}\# B$ is the unique positive definite solution of the equation $XAX=B$. As we have seen above, this is true in any twisted subgroup which is $2$-divisible. 

We also show that a unital $C^*$-algebra satisfying the identity $[[x,y],z]=0$ is necessarily commutative. As the class of split $C^*$-by-nilpotent algebras contains all unital $C^*$-algebras, L.~Molnár's commutativity and associativity results follow. Our method is motivated by the theory of analytic quasigroups, their local expansions and infinitesimal behavior (see~\cite{MR0262413,MR1125815} and the survey paper~\cite{MR1780879}).

\section{Split $C^*$-by-nilpotent algebras}

In this section, we define split $C^*$-by-nilpotent algebras. The motivation is the set $S$ of upper triangular $3\times 3$ matrices of the form
\begin{align} \label{eq:upper-triangular}
\begin{pmatrix}
a & c_{1} & c_{2} \\
0 & a & c_{3} \\
0 & 0 & a 
\end{pmatrix}, 
\end{align}
where $a \in \mathbb{R}^+$ and $c_1,c_2,c_3 \in \mathbb{C}$. Since $S$ is closed under addition, multiplication and taking square root, the midpoint operations \eqref{eq:circdef}, \eqref{eq:sharpdef}, \eqref{eq:stardef} are well-defined in $S$. However, the algebraic properties of $(S,\circ)$ are different from those of structures whose underlying set is formed by the positive elements of a $C^*$-algebra. For example, while $S$ is not commutative, for any $X,Y\in S$, we have
\[
X\circ Y = (\sqrt{X}+\sqrt{Y})^2,
\]
which implies that $(S,\circ)$ is a semigroup.

From now on, every $C^*$-algebra is supposed to be unital. This is not restrictive, since any non-unital $C^*$-algebra $\A$ admits a canonical \emph{unitisation} $\A_1 = \mathbb{C} \ltimes \A$ in which $\A$ embeds as a two-sided ideal~\cite{Rørdam_Larsen_Laustsen_2000}.

\begin{definition}\label{def_split}
Let $B$ be a unital associative algebra over $\mathbb{C}$. We say that $B$ is a \emph{split $C^*$-by-nilpotent} algebra if the following conditions hold:
\begin{enumerate}[(a)]
\item $B$ has a unital subalgebra $\mathcal{A}$ that is a $C^*$-algebra and $1_\A=1_B$.
\item $B$ has an ideal $N$ which consists of nilpotent elements;
\item $B=\A + N$; \label{direct sum}
\item The elements of $\mathcal{A}$ and $N$ commute.
\end{enumerate}
\end{definition}

\begin{remark}\label{remsplit}{\ }
\begin{enumerate}[(i)]
\item $\A \cap N =0$, since $\A \cap N$ is a commutative $C^*$-algebra consisting of nilpotent elements. Thus, $B=\A \oplus N$ as vector spaces.
\item $\A \cong B/N$ as associative algebras: indeed, the map $\pi_1 \colon b=a+n \mapsto a$ defines a unital associative algebra epimorphism $B \to \A$ with kernel $N$.
\item The diagram
\begin{equation*}\label{diag:SplV}
\begin{tikzcd}
N \arrow[r, "i"]
&B \arrow[r, shift left, "\pi_1"]
&\A \arrow[l, shift left, "j"]
\end{tikzcd}
\end{equation*}
where $i$ and $j$ are the canonical inclusions, is a split extension in the variety of associative algebras. Indeed, as shown above, $\pi_1$ is an epimorphism with kernel $N$. Moreover, since $\A$ is a subalgebra of $B$, the inclusion $j$ is a homomorphism of associative algebras such that $\pi_1\circ j=\operatorname{id}_{\A}$.
\item $N$ is the \emph{Jacobson radical} $\rad(B)$ of $B$. Indeed, $N\subseteq \rad(B)$ and $\rad(\mathcal{A})=0$ (see~\cite[Corollary 3.2.13]{Dales}).
\item The decomposition $B=\A \oplus N$ is unique: if $B=\A \oplus N = \A' \oplus N'$, then $N=\rad(B)=N'$ and $\A \cong B/N =B/ \rad(B) =B/N' \cong \A'$ as associative algebras. Thus, the $C^*$-algebra structure of $\A$ is uniquely determined.
\end{enumerate}
\end{remark}

We have two rich classes of examples for split $C^*$-by-nilpotent algebras.
\begin{example}{\ }
\begin{enumerate}[(i)]
\item Every $C^*$-algebra $\A$ can be seen as a split $C^*$-by-nilpotent algebra $\A=\A \oplus 0$.
\item Let $\mathcal{A}$ be a commutative $C^*$-algebra and let $N_0$ be a nilpotent $\mathbb{C}$-algebra. Define $N_1=\mathbb{C}\cdot 1 \oplus N_0$ and $B=\mathcal{A} \otimes_\mathbb{C} N_1$. Then $\mathcal{A}\otimes_\mathbb{C} 1$ is a subalgebra isomorphic to~$\mathcal{A}$, $\mathcal{A}\otimes_\mathbb{C} N_0$ is nilpotent, their elements commute and their direct sum is isomorphic to~$B$. 
\item Consider the subalgebra $B_3$ of $\operatorname{M}_3(\mathbb C)$ defined by
\[
B_3=
\left\{
\begin{pmatrix}
a & c_1 & c_2 \\
0 & a & c_3 \\
0 & 0 & a
\end{pmatrix}
\; \Bigg\vert \; a,c_1,c_2,c_3\in \mathbb C
\right\}.
\]
Then $B_3$ is a split $C^*$-by-nilpotent algebra. Indeed, set $\mathcal{A}=\mathbb{C} I_3$ and
\[
N=
\left\{
\begin{pmatrix}
0 & c_1 & c_2\\
0 & 0 & c_3\\
0 & 0 & 0
\end{pmatrix}
\; \Bigg\vert \; c_1,c_2,c_3\in \mathbb C
\right\}.
\]
Then, $\mathcal{A}$ is a unital subalgebra of $B_3$, which is a $C^*$-algebra isomorphic to
$\mathbb{C}$, and $1_{\mathcal{A}}=I_3=1_{B_3}$. Moreover, $N$ is an ideal of $B_3$ which consists of nilpotent elements. More precisely, $N^3=0$. Finally, every element of $B_3$ can be written uniquely as
\[
\begin{pmatrix}
a & c_1 & c_2\\
0 & a & c_3\\
0 & 0 & a
\end{pmatrix}
=
aI_3+
\begin{pmatrix}
0 & c_1 & c_2\\
0 & 0 & c_3\\
0 & 0 & 0
\end{pmatrix},
\]
so that $B_3=\mathcal{A} \oplus N$. Since $\mathcal A$ consists of scalar matrices, every element of $\mathcal A$ commutes with every element of $N$. Therefore, $B_3$ satisfies all the conditions in~\Cref{def_split}.
\end{enumerate}
\end{example}

From now on, when we speak of a split $ C^*$-by-nilpotent algebra $B=\mathcal{A}\oplus N$, then we mean that $\mathcal{A}$ is a unital $C^*$-algebra, $N$ consists of nilpotent elements, and $1=1_B=1_\A$. If we write an element $x\in B$ as $x=a+n$, then $a\in \mathcal{A}$ and $n\in N$.

We define the set $B^{+}$ of \emph{non-negative} elements and the set $B^{++}$ of \emph{positive} elements of a split $C^*$-by-nilpotent algebra $B=\mathcal{A}\oplus N$ as
\begin{align*}
B^+ &= \{a+n \in B \mid a\in \mathcal{A}^+, n\in N\},\\ 
B^{++} &= \{a+n \in B \mid a\in \mathcal{A}^{++}, n\in N\}.
\end{align*}
Furthermore, we denote by $1+N$ the set $\mathcal{U}(B)$ of invertible elements of $B$. It is easy to show that $1+N=\mathcal{U}(\mathcal{A}) \oplus N$. Indeed, if $x=a+n$ is invertible, then so is $a$. Conversely, assume that $a\in\mathcal{A}$ is invertible. Since $N$ is an ideal of $B$, we have $a^{-1}n\in N$, and there exists $k \in \mathbb{N}$ such that
\[
(a^{-1}n)^k=0.
\]
Hence $1+a^{-1}n$ is invertible, with inverse
\[
(1+a^{-1}n)^{-1}
=
\sum_{j=0}^{k-1}(-1)^j(a^{-1}n)^j.
\]
Indeed, if $x=a^{-1}n$, one has
\[
(1+x)\sum_{j=0}^{k-1}(-1)^j x^j
=
\sum_{j=0}^{k-1}(-1)^j x^j
+
\sum_{j=1}^{k}(-1)^{j-1}x^j
=
1+(-1)^{k-1}x^k
=
1.
\]
Similarly,
\[
\left(\sum_{j=0}^{k-1}(-1)^j x^j\right)(1+x)=1.
\]
Hence, $a+n=a(1+a^{-1}n)$ is invertible. Note that $B^{++}$ is contained in $1+N$.

\begin{example}
Consider the split $C^*$-by-nilpotent algebra
\[
B_3=
\left\{
\begin{pmatrix}
a & c_1 & c_2 \\
0 & a & c_3 \\
0 & 0 & a
\end{pmatrix}
\; \Bigg\vert \; a,c_1,c_2,c_3\in \mathbb{C}
\right\}.
\]
One may check that
\[
B_3^+=\left\{
\begin{pmatrix}
a & c_1 & c_2\\
0 & a & c_3\\
0 & 0 & a
\end{pmatrix}
\; \Bigg\vert \;
a \in \mathbb{R}^+ \cup \{0\}, \ c_1,c_2,c_3 \in \mathbb{C} \right\},
\]
while
\[
B_3^{++}=\left\{
\begin{pmatrix}
a & c_1 & c_2\\
0 & a & c_3\\
0 & 0 & a
\end{pmatrix}
\; \Bigg\vert \;
a \in \mathbb{R}^+, \ c_1,c_2,c_3 \in \mathbb{C}
\right\}
\]
is the set of matrices \eqref{eq:upper-triangular}. Finally, the set of invertible elements of $B_3$ is
\[
1+N=\left\{
\begin{pmatrix}
a & c_1 & c_2\\
0 & a & c_3\\
0 & 0 & a
\end{pmatrix}
\; \Bigg\vert \;
a,c_1,c_2,c_3 \in \mathbb{C}, \ a \neq 0
\right\}.
\]
\end{example}

\begin{lemma}\label{rem_square}
Let $B=\mathcal{A}\oplus N$ be a split $C^*$-by-nilpotent algebra. 
For any $x=a+n \in B^{++}$ there exists a unique square root of $x$ in
$B^{++}$.
\end{lemma}

\begin{proof}
Let $x=a+n\in B^{++}$ and let $b=\sqrt a$. Since $a\in \mathcal{A}^{++}$,
we have $b\in \mathcal{A}^{++}$ and $b$ is invertible. Moreover, since
$n$ is nilpotent, there exists $k\in\mathbb N$ such that $n^k=0$. Since
the elements of $\mathcal{A}$ and $N$ commute, we also have
\[
(b^{-2}n)^k=0.
\]
Thus, the series
\[
(1+t)^{1/2}=\sum_{j=0}^{\infty} \binom{1/2}{j} t^j,
\]
where
\[
\binom{1/2}{0}=1 \quad \text{ and } \quad
\binom{1/2}{j}=\frac{\left(\frac12\right)\left(\frac12-1\right)\cdots\left(\frac12-j+1\right)}{j!},\quad \text{for any } j\geq 1,
\]
becomes a finite sum when evaluated at $t=b^{-2}n$. We may therefore define
\[
y=b\sum_{j=0}^{k-1}\binom{1/2}{j}(b^{-2}n)^j .
\]
Then $y\in B^{++}$, and since $b$ commutes with $n$, we get
\[
y^2=b^2
\left(
\sum_{j=0}^{k-1}\binom{1/2}{j}(b^{-2}n)^j
\right)^2
=b^2(1+b^{-2}n)=a+n=x.
\]
Hence $y$ is a square root of $x$ in $B^{++}$. The uniqueness of $y$ follows from the uniqueness of the positive square root in the quotient $B/N\cong \mathcal{A}$.
\end{proof}

\Cref{remsplit} allows us to define the notion of \emph{morphism} of split $C^*$-by-nilpotent algebras.

\begin{definition}
Let $B=\A \oplus N$ and $B'=\A' \oplus N'$ be split $C^*$-by-nilpotent algebras. A \emph{morphism} $B \to B'$ is a morphism of associative algebra split extensions
\begin{equation*}
\begin{tikzcd}
N \arrow [r, "i"] \ar[d, "f"']
&B \arrow[r, shift left, "\pi_1"] \ar[d, swap, "g"]&
\A \ar[l, shift left, "j"] \ar[d, "h"] \\
N' \arrow [r, "i'"]
&B' \arrow[r, shift left, "\pi'"] &
\A' \ar[l, shift left, "j'"] 
\end{tikzcd}
\end{equation*}
such that $h$ is a unital $^*$-homomorphism. In other words, $f$, $g$, and $h$ are homomorphisms of associative algebras, $h$ is moreover a unital $^*$-homomorphism, and the following identities hold:
\[
g\circ i=i'\circ f,\qquad
h\circ \pi_1=\pi'_1\circ g,\qquad
g\circ j=j'\circ h.
\]
\end{definition}

We observe that the morphism $g \colon B \to B'$ is uniquely determined by $f$ and $h$. In fact, any $b \in B$ can be written in a unique way as $b=a+n$, with $a \in \A$ and $n \in N$. Thus, $g(b)=g(a+n)=h(a)+f(n)$. Moreover, $g(1_B)=h(1_{\A})=1_{\A'}=1_{B'}$.

We denote by \textbf{Spl$C^*$-Nil} the category whose objects are the split $ C^*$-by-nilpotent algebras and whose morphisms are those defined above. This defines an \emph{adjunction}~\cite{MacLane} between $\textbf{Spl$C^*$-Nil}$ and the category \textbf{$C^*$Alg$_1$} of unital $C^*$-algebras: the left adjoint of the full inclusion functor $i \colon \textbf{$C^*$Alg$_1$} \to \textbf{Spl$C^*$-Nil}$ is the functor $p \colon \textbf{Spl$C^*$-Nil} \to \textbf{$C^*$Alg$_1$}$ which sends every split $C^*$-by-nilpotent algebra $B=\A \oplus N$ to the unital $C^*$-algebra $\A$, and every morphism $(f,h) \colon B \to B'$ as above to the unital $^*$-homomorphism $h \colon \A \to \A'$.

Indeed, for every object $\A \oplus N$ of $\textbf{Spl$C^*$-Nil}$ and $\A'$ of \textbf{$C^*$Alg$_1$}, and for any unital $*$-homomorphism $h \colon \A \to \A'$, the map $h \mapsto (h,0_N)$ defines a bijection
\[
\operatorname{Hom}_{C^*\textbf{Alg}_1}(\A,\A') \cong \operatorname{Hom}_{\textbf{Spl}C^*\textbf{-Nil}}(\A \oplus N, \A' \oplus 0)
\]
which is natural in both $\A \oplus N$ and $\A'$.

\medskip

We now define the notion of \emph{semi-norm} for a split $C^*$-by-nilpotent algebra. We recall the following definition.

\begin{definition}~\cite{MR1741419}
A \emph{semi-norm} on a $\mathbb{C}$-vector space $V$ is a function $p \colon V \to \mathbb{R}$ such that 
\begin{enumerate}[(a)]
\item $p(x+y) \leq p(x)+p(y)$;
\item $p(\lambda x) = |\lambda| p(x)$;
\end{enumerate}
for any $x,y\in V$ and $\lambda \in \mathbb{C}$. 
\end{definition}

For a split $C^*$-by-nilpotent algebra $B=\mathcal{A}\oplus N$, we define the semi-norm $\sn{\cdot}$ of $x=a+n\in B$ by
\[
\sn{x}=\|a\|.
\]
This semi-norm is submultiplicative as well: $\sn{xy} \leq \sn{x}\sn{y}$. Moreover, if $\sn{1-x}<1$, then $x$ is invertible.

Let $B=\mathcal{A}\oplus N$ be a split $C^*$-by-nilpotent algebra. Let $\mathbb{C}[[\tau]]$ denote the ring of formal power series over $\mathbb{C}$ in the variable $\tau$. Let $u(\tau)\in \mathbb{C}[[\tau]]$ and assume that $u(\tau)$ has convergence radius $\rho>0$ at $0$. If $x=a+n\in B$ with $\sn{x}<\rho$, then the higher derivatives $u^{(i)}(\tau)$ are absolutely convergent at $a$ and
\begin{align} \label{eq:infinitesum}
u(x)=u(a+n)=\sum_{i=0}^{M-1}\frac{u^{(i)}(a)}{i!} n^i,
\end{align}
is a well-defined element of $B$, if $M$ is the smallest positive integer such that $n^M=0$. In particular, if $u(\tau)$ is a polynomial, then $x\mapsto u(x)$ is the usual substitution. 

Let $R_\rho$ denote the set of formal power series whose convergence radius at $0$ is at least~$\rho$. The set $R_\rho$ is a commutative ring and the assignment $u \mapsto u(x)$ defines a ring homomorphism $R_\rho \to B$, for any $x\in B$ such that $\sn{x}<\rho$. This is obvious for the sum $u+v$ of formal power series. For the product $w=uv$ with $u,v \in R_\rho$, this follows from
\begin{align*}
u(a+n)v(a+n)&=\left(\sum_{i=0}^{M-1}\frac{u^{(i)}(a)}{i!} n^i\right) \left(\sum_{j=0}^{M-1}\frac{v^{(j)}(a)}{j!} n^j\right) \\
&= \sum_{k=0}^{M-1} \left(\sum_{i=0}^{k} \frac{u^{(i)}(a)v^{(k-i)}(a)}{i!(k-i)!}\right)n^k\\
&= \sum_{k=0}^{M-1}\frac{w^{(k)}(a)}{k!} n^k.
\end{align*}
Thus, a formal algebraic identity of formal power series holds after substitutions of values from $B$, with semi-norms less than $\rho$. 

\begin{lemma} \label{lm:seminorm}
Let $B=\mathcal{A}\oplus N$ be a split $C^*$-by-nilpotent algebra. Let $u(\tau)\in \mathbb{C}[[\tau]]$ be a formal power series and assume that $u(\tau)$ has convergence radius $\rho >0$ at $0$. If $x,y\in B$, $\sn{x}<\rho$, and $xy=yx$, then $u(x)$ and $y$ commute.
\end{lemma}

\begin{proof}
Let $x=a+n, y=a'+n' \in B$. If $(a+n)(a'+n')=(a'+n')(a+n)$, then $aa'=a'a$ and $nn'=n'n$. Hence, $a'$ and $n'$ commute with $u^{(i)}(a)$ and $n^i$ for any $0 \leq i \leq M-1$ respectively, which implies that $u(a+n)$ and $a'+n'$ commute, since $\sn{x} < \rho$. 
\end{proof}

\section{Midpoint operations on $B^{++}$}

In this section, we aim to show that the operations \eqref{eq:circdef}, \eqref{eq:sharpdef}, and \eqref{eq:stardef} are well-defined on the set of positive elements of a split $C^*$-by-nilpotent algebra, and \eqref{eq:circdef} is commutative. Our principal tool is the formal power series
\begin{equation} \label{eq:formalsroot}
s(\xi)=(1+\xi)^{\frac{1}{2}}=1+{\frac {1}{2}}\xi-{\frac {1}{8}}{\xi}^{2}+{\frac {1}{16}}{\xi}^{3}-{\frac {5}{128}}{\xi}^{4}+\cdots
\end{equation}
i.e., the MacLaurin series expansion of $s(x)=(1+x)^{\frac{1}{2}}$. At $\xi=0$, $s(\xi)$ has convergence radius $1$. The formal identity $s(\xi)^2=1+\xi$ implies $s(x)^2=1+x$, for any $x\in B$ such that $\sn{x}<1$. Furthermore, for commuting variables $\xi,\eta$, the formal identity $s(\xi)s(\eta)=s(\xi+\eta+\xi\eta)$ holds.

\begin{lemma} \label{lm:sqrt}
Let $B=\mathcal{A}\oplus N$ be a split $ C^*$-by-nilpotent algebra and let $s(\xi)$ be the formal power series as in \eqref{eq:formalsroot}. The following hold:
\begin{enumerate}[(i)]
\item If $x\in B$ with $\sn{x}<1$, then $\sn{1-s(x)}\leq 1-\sqrt{1-\sn{x}}$.
\item If $x,y\in B$ with $\sn{x} < 1$ and $\sn{1-y}<1$, then $y^2=1+x$ holds if and only if $y=s(x)$.
\item There exists $\varepsilon>0$ such that, for each pair of commuting elements $x,y\in B$ with $\sn{x} < \varepsilon$ and $\sn{y} < \varepsilon$, one has $s(x)s(y)=s(x+y+xy)$. 
\end{enumerate}
\end{lemma}
\begin{proof}
\begin{enumerate}[(i)]
\item Let $x\in B$ with $\sn{x}<1$. Thus
\begin{align*}
\sn{1-s(x)} &= \bigsn{{\frac {1}{2}}(-x)+{\frac {1}{8}}{(-x)}^{2}+{\frac {1}{16}}{(-x)}^{3}+{\frac {5}{128}}{(-x)}^{4}+\cdots} \\
&\leq {\frac {1}{2}}\sn{-x}+{\frac {1}{8}}{\sn{-x}}^{2}+{\frac {1}{16}}{\sn{-x}}^{3}+{\frac {5}{128}}{\sn{-x}}^{4}+\cdots \\
&= 1-s(-\sn{-x}) \\
&= 1-\sqrt{1-\sn{x}}.
\end{align*}
\item Let $y\in {B}$ such that $\sn{1-y}<1$. If $y=s(x)$, then trivially $y^2=1+x$. Conversely, if $y^2=1+x$, then 
\[
\sn{1-s(x)}\leq 1-\sqrt{1-\sn{x}}<1,
\]
and
\[
\sn{2-s(x)-y} \leq \sn{1-s(x)}+\sn{1-y}<2,
\]
which implies $\sn{1-\frac{s(x)+y}{2}}<1$. Hence, $s(x)+y$ is invertible in ${B}$. As $y$ commutes with $x=y^2-1$, and $y$ commutes with $s(x)$ by \Cref{lm:seminorm}, then $y^2=s(x)^2$ implies $0=(y-s(x))(y+s(x))$. Since the second factor is invertible, we have $y-s(x)=0$. 
\item We can choose $\varepsilon>0$ \footnote{This is always true for $\varepsilon < \sqrt{2}-1$.} such that $\sn{x+y+xy}<1$ and $\sn{1-s(x)s(y)}<1$, for any $\sn{x} < \varepsilon$ and $\sn{y}<\varepsilon$. By \Cref{lm:seminorm}, $z_1=s(x)$, $z_2=s(y)$ and $z_3=s(x+y+xy)$ commute since $xy=yx$. We saw that $z_1^2=1+x$, $z_2^2=1+y$ and 
\[
z_3^2=1+x+y+xy=(1+x)(1+y)=z_1^2z_2^2=(z_1z_2)^2.
\]
Thus, $z_3=z_1z_2$ follows from (ii).  
\end{enumerate}
\end{proof}

\begin{lemma} \label{lm:sqrtB+}
Let $B=\mathcal{A}\oplus N$ be a split $C^*$-by-nilpotent algebra. For any $a\in \mathcal{A}$ and $n\in N$, the following hold:
\begin{enumerate}[(i)]
\item If $n\in N$, then $s(n)\in 1+N$ and $s(n)^2=1+n$.
\item The map $n\mapsto s(n)$ is a bijection from $N$ to $1+N$. In particular, any element $y\in 1+N$ has a unique square root $\sqrt{y}=s(y-1) \in 1+N$.
\item Every element of $B^{++}$ has a unique square root in $B^{++}$. 
\end{enumerate}
\end{lemma}
\begin{proof}
\begin{enumerate}[(i)]
\item It is an immediate consequence of \Cref{rem_square}.
\item For any $n\in N$, $s(n^2+2n)=s(n)^2=1+n$ by (iii) of \Cref{lm:sqrt} with $x=y=n$. Hence, for $y\in 1+N$, we have $s(y^2-1)=y$, which means that the map $x\mapsto s(x)$ is surjective. The injectivity follows from
\[
s(n_1)=s(n_2) \Rightarrow 1+n_1=s(n_1)^2=s(n_2)^2=1+n_2 \Rightarrow n_1=n_2.
\]
\item Fix an element $a+n\in B^{++}$. As $a$, $\sqrt{a}$ and $n$ commute, and $a^{-1}n\in N$, we have
\[
(\sqrt{a}\, s(a^{-1}n))^2 = a\, s(a^{-1}n)^2 =a(1+a^{-1}n)=a+n.
\]
The uniqueness follows from the uniqueness of the square root in $\mathcal{A}^{++}$ and in $1+N$. 
\end{enumerate}
\end{proof}

We are now ready to prove the following.

\begin{proposition} \label{pr:B_oper_welldef}
The operations \eqref{eq:circdef}, \eqref{eq:sharpdef} and \eqref{eq:stardef} are well-defined on $B^{++}$. Furthermore, the operation \eqref{eq:circdef} is commutative on $B^{++}$. 
\end{proposition}
\begin{proof}
It follows easily from \Cref{lm:sqrtB+} that the operations \eqref{eq:circdef}, \eqref{eq:sharpdef}, and \eqref{eq:stardef} are well-defined on $B^{++}$. It remains to show that $x \circ y = y \circ x$, for any $x,y \in B^{++}$.

Since $B^{++}$ is a $2$-divisible twisted subgroup of the group of units $1+N$, and
\begin{align} \label{eq:circbysharp}
x\circ y =x+y+x(x^{-1}\# y) +(x^{-1}\# y)x,
\end{align}
with
\begin{align} \label{8++}
(x(x^{-1}\# y))^2=xy, \quad ((x^{-1}\# y)x)^2=yx,
\end{align}
holds, it follows that \eqref{eq:circdef} is commutative on $B^{++}$.
\end{proof}

\section{Commutator brackets}

The purpose of this section is to prove the ``if'' part of \Cref{thm:main} and to show that in any unital $C^*$-algebra, the identity $[[x,y],y]=0$ implies that the algebra is commutative. We start with the following lemma, whose proof is a linear algebra exercise.

\begin{lemma} \label{lm:jacobi}
Let $A$ be an associative algebra over the field $\F$, with $\mathrm{char}(\F)\neq 2,3$. Let $S$ be a vector subspace of $A$ that is closed under the ternary Lie bracket $[[-,-],-]$. Assume that $[[x,y],y]-[[y,x],x]=0$, for any $x,y\in S$. Then $[[x,y],z]=0$, for any $x,y,z\in S$. 
\end{lemma}
\begin{proof}
Linearizing $[[x,y],y]-[[y,x],x]=0$ in $y$, since $\mathrm{char}(\F)\neq 2$, we get $[[x,y],z]+[[x,z],y]=0$ for any $x,y,z\in S$. Hence
\[
[[x,y],z]=[[z,x],y]=[[y,z],x].
\]
We substitute this in the Jacobi identity
\[
[[x,y],z]+[[y,z],x]+[[z,x],y]=0
\]
to obtain $3[[x,y],z]=0$, for any $x,y,z\in S$. Thus, $[[x,y],z]=0$ holds in $S$. 
\end{proof}

A similar result can be obtained for the class of split $C^*$-by-nilpotent algebras.

\begin{lemma} \label{lm:lts}
Let $B=\mathcal{A}\oplus N$ be a split $C^*$-by-nilpotent algebra. Assume that for some $\varepsilon>0$, the identity $[[x,y],y]-[[y,x],x]=0$ holds for any $x,y\in B^{++}$ such that $\sn{x}< \varepsilon$, $\sn{y}<\varepsilon$. Then $[[x,y],z]=0$, for any $x,y,z \in B$. 
\end{lemma}
\begin{proof}
We denote
\[
\Sigma = \{a+n\in B \mid \text{$a$ is self-adjoint}\}
\]
and we fix arbitrary elements $x,y\in \Sigma$. Then there exist real numbers $\alpha,\beta > 0$ such that 
\[
x'=\alpha x+\beta\cdot 1, \quad y'=\alpha y+\beta\cdot 1\in B^{++}
\]
and $\sn{x'}< \varepsilon$, $\sn{y'}<\varepsilon$. By assumption and the properties of the Lie bracket, we have
\[
0=[[x',y'],y']-[[y',x'],x']=\alpha^3([[x,y],y]-[[y,x],x]).
\]
As $\Sigma$ is an $\mathbb{R}$-vector space, \Cref{lm:jacobi} implies that $[[x,y],z]=0$, for any $x,y,z\in \Sigma$. It is well known that any element of a $C^*$-algebra can be written as $a+ib$ with $a,b$ self-adjoint. In this way, any element of $B$ can be written as $x+iy$, with $x,y\in \Sigma$. This implies that $[[x,y],z]=0$, for any $x,y,z\in B$. 
\end{proof}

\begin{remark}
As a direct consequence of \Cref{lm:lts}, one may easily see that the following are equivalent.
\begin{enumerate}[(i)]
\item The identity $[[x,y],z]=0$ holds in $B$;
\item The identity $[[x,y],y]-[[y,x],x]=0$ holds in $B^{++}$;
\item There exists $\varepsilon >0$ such that $[[x,y],y]-[[y,x],x]=0$ holds for any $x,y\in B^{++}$ with $\sn{x}< \varepsilon$, $\sn{y}<\varepsilon$.
\end{enumerate}
\end{remark}

We are now ready to prove the following.

\begin{proposition} \label{pr:lts}
Let $B=\mathcal{A}\oplus N$ be a split $C^*$-by-nilpotent algebra. If the identity $[[x,y],y]-[[y,x],x]=0$ holds in $B^{++}$, then
\[
x\circ y = (\sqrt{x}+\sqrt{y})^2,
\]
for any $x,y \in B^{++}$.
\end{proposition}

\begin{proof}
One may check that, in any ring, the identities
\begin{align}\label{eq:pr:lts_1}
ab^2a-ba^2b = [[a,b],b]a-b[[b,a],a]
\end{align}
and
\begin{align}\label{eq:pr:lts_2}
a^2b^2a^2 - (aba)^2 = a(ab^2a-ba^2b)a
\end{align}
hold. By \Cref{lm:lts} the identity $[[x,y],z]=0$ holds for any $x,y,z\in B^{++}$, since we assumed $[[x,y],y]-[[y,x],x]=0$. This implies that equation \eqref{eq:pr:lts_1} becomes $ab^2a-ba^2b=0$, hence \eqref{eq:pr:lts_2} leads to $\sqrt{xyx}=\sqrt{x}\sqrt{y}\sqrt{x}$. Thus,
\begin{align*}
x\circ y &= \frac{1}{\sqrt{x}}\left(x+x^\frac{1}{4}y^\frac{1}{2}x^\frac{1}{4}\right)^2\frac{1}{\sqrt{x}} \\
&=  \left(x^{-\frac{1}{4}}\left(x+x^\frac{1}{4}y^\frac{1}{2}x^\frac{1}{4}\right)x^{-\frac{1}{4}}\right)^2 \\
&= (\sqrt{x}+\sqrt{y})^2, \qedhere
\end{align*}
where the second equality follows from $a^2b^2a^2-(aba)^2=0$ and the third one follows by straightforward computations.
\end{proof}

\begin{corollary}\label{cor_if}
If the identity $[[x,y],z]=0$ holds in $B$, then $(B^{++},\circ)$ is a semigroup. \noproof
\end{corollary}

An immediate consequence of \Cref{pr:lts} is that, if the identity $[[x,y],y]=0$ is satisfied by all positive elements of a $C^*$-algebra $\mathcal{A}$, then $\mathcal{A}$ is commutative.

\begin{proposition} \label{pr:Cstar_lts}
Let $\mathcal{A}$ be a $C^*$-algebra such that $[[x,y],y]=0$ for any $x,y \in \mathcal{A}^{++}$. Then $\mathcal{A}$ is commutative.
\end{proposition}
\begin{proof}
By the Kleinecke-Shirokov Theorem (see~\cite{MR0087914} and~\cite{Shi56}), $[[x,y],y]=0$ implies that $[x,y]$ is quasi-nilpotent, and thus, $\spec({[x,y]})=\{0\}$. If $x,y \in \mathcal{A}^{++}$, then $i[x,y]$ is self-adjoint, and $\|[x,y]\|=\|i[x,y]\|$ is equal to the spectral radius. This means that $[x,y]=0$ holds in $\A^{++}$. Finally, if all positive elements commute, then all self-adjoint elements do so, and thus, we obtain the commutativity of $\mathcal{A}$. 
\end{proof}

As mentioned in the introduction, \Cref{pr:Cstar_lts} makes clear that L.~Molnár's commutativity and associativity results~\cite{molnar2021} are consequences of \Cref{thm:main}.

\section{Power series expansion of the associator}

To prove the converse implication of \Cref{cor_if}, and so to conclude the proof of \Cref{thm:main}, we need to compute the low-degree terms of the power series expansion of the operation \eqref{eq:circdef}. 

Let $\xi,\eta$ be non-commuting variables over $\mathbb{C}$. Define the formal power series 
\begin{align*}
\hat{s}(\xi,\eta)=s(\xi+\eta+\xi\eta)+s(\xi+\eta+\eta\xi)-2-\xi-\eta,
\end{align*}
where $s$ is defined as in \eqref{eq:formalsroot}. If $B$ is a split $C^*$-by-nilpotent algebra and $x,y\in B$ with $\sn{x},\sn{y}<\sqrt{2}-1$, then one may see that $\sn{x+y+xy}, \sn{x+y+yx} < 1$, and $\hat{s}(x,y)$ is a well-defined element of $B$. 

\begin{lemma} \label{lm:prodrule}
There exists $\varepsilon>0$ such that for any $c,d>0$ and $x,y\in B$ with $\sn{x}<c\varepsilon$, $\sn{y}<d\varepsilon$, we have
\begin{equation} \label{eq:prodrule}
(c+x)\circ (d+y) = (\sqrt{c}+\sqrt{d})^2+(\sqrt{c}+\sqrt{d})\left(\frac{x}{\sqrt{c}}+\frac{y}{\sqrt{d}}\right) + \sqrt{cd} \, \hat{s}\left(\frac{x}{c},\frac{y}{d}\right)
\end{equation}
and
\begin{equation}\label{eq_lemma5.1}
\sn{(c+x)\circ (d+y) - (\sqrt{c}+\sqrt{d})^2}\leq 3(\sqrt{c}+\sqrt{d})^2\varepsilon.
\end{equation}
\end{lemma}
\begin{proof}
We can choose $\varepsilon >0$ small enough\footnote{It is enough to choose again $\varepsilon < \sqrt{3-2\sqrt{2}}=\sqrt{2}-1$.} such that, if $\sn{x-1},\sn{y-1}<\varepsilon$ then
\[
\sn{x(x^{-1}\# y)-1}, \sn{(x^{-1}\# y)x-1}, \sn{xy-1}, \sn{yx-1} < 1.
\]
If $x,y\in B^{++}$, then by (ii) of \Cref{lm:sqrt} and equation \eqref{8++}, we have
\[
x(x^{-1}\# y)=s(xy-1), \qquad (x^{-1}\# y)x=s(yx-1).
\]
In other words, if $x,y\in B$ are such that $\sn{x},\sn{y}<\varepsilon$, then $1+x,1+y\in B^{++}$, and
\begin{align*}
(1+x)((1+x)^{-1}\# (1+y)) &= s(x+y+xy),\\
((1+x)^{-1}\# (1+y))(1+x) &= s(x+y+yx).
\end{align*}
We take now $c,d>0$ and $\sn{x}<c\varepsilon$, $\sn{y}<d\varepsilon$. Then
\begin{align*}
(c+x)((c+x)^{-1}\# (d+y)) &= \sqrt{cd} \, s\left(\frac{x}{c} + \frac{y}{d} + \frac{xy}{cd}\right),\\
((c+x)^{-1}\# (d+y))(c+x) &= \sqrt{cd} \, s\left(\frac{x}{c} + \frac{y}{d} + \frac{yx}{cd}\right).
\end{align*}
By equation \eqref{eq:circbysharp}, one has
\begin{align*} \label{eq:cxody}
(c+x)\circ (d+y) &= c+d+x+y+\sqrt{cd} \left(s\left(\frac{x}{c} + \frac{y}{d} + \frac{xy}{cd}\right) + s\left(\frac{x}{c} + \frac{y}{d} + \frac{yx}{cd}\right)
\right) \\
&= c+d+x+y+\sqrt{cd} \left(\hat{s}\left(\frac{x}{c},\frac{y}{d}\right) +2 +\frac{x}{c} + \frac{y}{d} \right) \\
&= (\sqrt{c}+\sqrt{d})^2+(\sqrt{c}+\sqrt{d})\left(\frac{x}{\sqrt{c}}+\frac{y}{\sqrt{d}}\right) + \sqrt{cd} \, \hat{s}\left(\frac{x}{c},\frac{y}{d}\right).
\end{align*}
This proves \eqref{eq:prodrule}. Furthermore
\begin{align*}
\sn{(c+x)\circ (d+y) - (\sqrt{c}+\sqrt{d})^2} &= \sn{x+y+\sqrt{cd} \left(s\left(\alpha \right)+s\left(\beta\right)-2\right)} \\
&\leq \sn{x} + \sn{y} + \sqrt{cd}(\sn{\alpha} + \sn{\beta }) \\
&< \varepsilon c + \varepsilon d + (4\varepsilon + 2\varepsilon^2)\sqrt{cd} \\
&< 3(\sqrt{c}+\sqrt{d})^2\varepsilon,
\end{align*}
where $\alpha=\dfrac{x}{c}+\dfrac{y}{d}+\dfrac{xy}{cd}$ and $\beta=\dfrac{x}{c}+\dfrac{y}{d}+\dfrac{yx}{cd}$. Thus, \eqref{eq_lemma5.1} holds.
\end{proof}

Our goal now is to compute the low-degree terms of the difference
\[
A(x,y) = ((1+x) \circ 1) \circ (1+y) - (1+x) \circ (1 \circ (1+y)),
\]
where $x,y \in B^{++}$ are elements of small semi-norm. We start with a closer look at the formal power series $\hat{s}(\xi,\eta)$.
\begin{lemma} \label{lm:shatlow}
The low degree terms of the formal power series $\hat{s}(\xi,\eta)$ are given by
\begin{equation} \label{eq:shat_terms}
\hat{s}(\xi,\eta)=-\frac{1}{4}(\xi-\eta)^2+\frac{1}{8}(\xi^3-\xi\eta\xi-\eta\xi\eta+\eta^3) +O(4),
\end{equation}  
where $O(4)$ is for a formal sum of monomials of degree at least $4$ in the non-commuting variables $\xi, \eta$.
\end{lemma}
\begin{proof}
One may check that
\begin{multline*}
s(\xi+\eta+\xi\eta)=1+\frac{1}{2}\,(\xi+\eta)-\frac{1}{8}\,(\xi^2-3\xi\eta+\eta\xi+\eta^2)+ \\ \frac{1}{16}\,
(\xi^3-\xi^2\eta-\xi\eta\xi-\xi\eta^2+\eta\xi^2-\eta\xi\eta+\eta^2\xi+\eta^3) +O(4)
\end{multline*}
and then
\begin{multline*} \label{eq:s_sym_terms}
s(\xi+\eta+\xi\eta)+s(\xi+\eta+\eta\xi)= \\
2+\xi+\eta-\frac{1}{4}\,(\xi-\eta)^2 + \frac{1}{8}\,
(\xi^3-\xi\eta\xi-\eta\xi\eta+\eta^3) +O(4),
\end{multline*}
which implies \eqref{eq:shat_terms}.
\end{proof}

From now on, the notation $O(4)$ is for a (possibly infinite) sum of monomials of degree at least $4$ in the non-commuting variables $x,y$.

\begin{proposition} \label{pr:Axy_expansion}
Let $B=\mathcal{A}\oplus N$ be a split $C^*$-by-nilpotent algebra. There exists $\varepsilon>0$ such that, for any $x,y\in B$ with $\sn{x}< \varepsilon$ and $\sn{y}<\varepsilon$, the low degree terms of $A(x,y)$ are given by
\begin{align*}
A(x,y) = \dfrac{1}{32}([[x,y],y]-[[y,x],x]) + O(4).
\end{align*}
\end{proposition}
\begin{proof}
We define
\[
B(x,y) = ((1+x) \circ 1) \circ (1+y).
\]
As $1+x,1+y \in B^{++}$ and $(1+x) \circ (1 \circ (1+y)) = ((1+y) \circ 1) \circ (1+x) = B(y,x)$ for the commutativity of \eqref{eq:circdef}, we have
\[
A(x,y)=B(x,y)-B(y,x).
\]
We choose $\varepsilon>0$ such that \Cref{lm:prodrule} holds for $\varepsilon'=3\varepsilon$ and we define $w$ by
\[
(1+x)\circ 1 = 4 +2x +\hat{s}(x,0) = 4+w.
\]
By \eqref{eq_lemma5.1}, $\sn{w}<12\varepsilon=4\varepsilon'$ and we can apply again \Cref{lm:prodrule} with $1+y$ and $4+w$. Thus, by \eqref{eq:prodrule} we obtain
\begin{align*}
B(x,y) &= (4+w)\circ (1+y) \\
&= 9+3\left(\frac{w}{2}+y\right) +2\hat{s}\left(\frac{w}{4},y\right) \\
&= 9+3(x+y) + \frac{3}{2}\hat{s}(x,0) + 2\hat{s}\left(\frac{x}{2}+\frac{1}{4}\hat{s}(x,0),y\right). 
\end{align*}
By \Cref{lm:shatlow}, we have that the low degree terms of $B(x,y)$ are given by
\begin{multline*} \label{eq:Bxy}
B(x,y)=9+3(x+y)-\frac{1}{4}(2x^2-xy-yx+2y^2)+ \\ 
\frac{1}{32}(8x^3-x^2y-2xyx-yx^2-4yxy+8y^3) + O(4).
\end{multline*}  
As the conditions for $x$ and $y$ were the same, we can interchange them to obtain
\begin{align*}
A(x,y)& = \frac{1}{32}(xy^2-2yxy+y^2x-x^2y+2xyx-yx^2)+O(4)\\
&=\dfrac{1}{32}([[x,y],y]-[[y,x],x]) + O(4). \qedhere
\end{align*}

\end{proof}

We are now ready to prove the converse of \Cref{pr:lts}.

\begin{proposition}\label{prop17}
Let $B=\A \oplus N$ be a split $C^*$-by-nilpotent algebra. If $(B^{++}, \circ)$ is a semigroup, then $[[x,y],y]-[[y,x],x]=0$ for every $x,y \in B^{++}$ with $\sn{x}<\varepsilon$ and $\sn{y}<\varepsilon$, for some $\varepsilon > 0$.
\end{proposition}

\begin{proof}
Assume that the operation \eqref{eq:circdef} is associative on $B^{++}$. It is sufficient to show that $[[x,y],y]-[[y,x],x]=0$
holds whenever $x,y\in N$, or $x,y\in \mathcal{A}^{++}$ with $\|x\|<\varepsilon$ and $\|y\|<\varepsilon$, for some $\varepsilon>0$. As the elements of $\mathcal{A}$ and $N$ commute, this implies $[[x,y],y]-[[y,x],x]=0$, for every $x,y\in B^{++}$ with $\sn{x}<\varepsilon$ and $\sn{y}<\varepsilon$.

Let us choose $\varepsilon>0$ according to \Cref{lm:prodrule} and fix the elements $x,y\in B^{++}$ such that $\sn{x}< \varepsilon$ and $\sn{y}<\varepsilon$. If $\varepsilon'=\dfrac{\varepsilon}{\max\lbrace \sn{x},\sn{y}\rbrace}$, then by \Cref{pr:Axy_expansion}, for any $0<t<\varepsilon'$ we have that $\sn{tx} < \varepsilon$, $\sn{ty} < \varepsilon$, and
\begin{equation}
0=A(tx,ty)=\sum_{i=3}^{\infty} z_i t^{i},
\end{equation}
where $z_3=\dfrac{1}{32}([[x,y],y]-[[y,x],x])$, and $z_i$ is a homogeneous polynomial of degree~$i$ in the non-commuting variables $x,y$.

We prove the proposition by showing $z_3=0$. Let us first assume that $z_3\neq 0$ and $x,y\in N$. Let $k$ be the least positive integer such that $z_3\in N^k\setminus N^{k+1}$. For infinitely many choices of $t$, we have 
\begin{equation*}
z_3+z_4t+\cdots + z_{k} t^{k-3} \in N^{k+1},
\end{equation*}
which implies $z_3, z_4, \ldots, z_k \in N^{k+1}$, a contradiction.

Now, let $x,y\in \mathcal{A}^{++}$. Then $\|x\|< \varepsilon$, $\|y\|<\varepsilon$, the series 
\[
\sum_{i=4}^{\infty} z_i t^{i-3}
\]
is absolutely convergent and it tends to $0$ if $t\to 0$. Hence, we conclude that 
\[
z_3=-\sum_{i=4}^{\infty} z_i t^{i-3}=0.
\]
\end{proof}

We now have all the ingredients to prove \Cref{thm:main}.

\begin{proof}[Proof of \Cref{thm:main}]

Let $B=\A \oplus N$ be a split $C^*$-by-nilpotent algebra. By \Cref{pr:B_oper_welldef}, we have that the operation \eqref{eq:circdef} is well-defined and commutative on $B^{++}$.

Furthermore, if we suppose that $(B^{++}, \circ)$ is a semigroup, then by \Cref{prop17} and \Cref{lm:lts} the identity $[[x,y],z]=0$ holds in $B$.

Conversely, if $[[x,y],z]=0$ for any $x,y,z \in B$, then by \Cref{cor_if} we have that $x \circ y=(\sqrt{x}+\sqrt{y})^2$, and \eqref{eq:circdef} is associative on $B^{++}$.
\end{proof}

\section*{Acknowledgments}

The first, second and third authors are supported by the ``National Group for Algebraic and Geometric Structures and their Applications'' (GNSAGA -- INdAM). The first author is also supported by the University of Salento. The second and third authors are also supported by the University of Palermo, and by the SDF Sustainability Decision Framework Research Project -- MISE decree of 31/12/2021 (MIMIT Dipartimento per le politiche per le imprese -- Direzione generale per gli incentivi alle imprese) -- CUP:~B79J23000530005, COR:~14019279, Lead Partner:~TD Group Italia Srl, Partner:~Università degli Studi di Palermo. The second author is also supported by the University of Palermo (FFR2026, UNIPA BsD). The third author is also a postdoctoral researcher of the Fonds de la Recherche Scientifique--FNRS. The fourth author is supported by the National Research, Development and Innovation Fund of the Ministry for Innovation and Technology of Hungary under grant SNN 152582. 

\printbibliography

@book{Dales,
	title =     {Banach Algebras and Automatic Continuity},
	author =    {Dales, H. G.},
	publisher = {Oxford University Press, USA},
	isbn =      {9780198500131},
	year =      {2001},
	series =    {London Mathematical Society Monographs New Series}}

@book{Rørdam_Larsen_Laustsen_2000, 
place={Cambridge}, 
series={London Mathematical Society Student Texts}, 
title={An Introduction to K-Theory for C*-Algebras}, 
publisher={Cambridge University Press}, author={Rørdam, M. and Larsen, F. and Laustsen, N.}, 
year={2000}, 
collection={London Mathematical Society Student Texts}}

@article {molnar2021,
    AUTHOR = {Moln\'ar, L.},
     TITLE = {Maps on positive definite cones of {$C^*$}-algebras preserving the {W}asserstein mean},
   JOURNAL = {Proceedings of the American Mathematical Society},
  FJOURNAL = {Proceedings of the American Mathematical Society},
    VOLUME = {150},
      YEAR = {2022},
    NUMBER = {3},
     PAGES = {1209--1221},
   MRCLASS = {47A64 (46L40 47B49)},
  MRNUMBER = {4375715},
MRREVIEWER = {Tomohiro\ Hayashi},
       DOI = {10.1090/proc/15725}
}

@article {MR0087914,
	AUTHOR = {Kleinecke, D. C.},
	TITLE = {On operator commutators},
	JOURNAL = {Proceedings of the American Mathematical Society},
	FJOURNAL = {Proceedings of the American Mathematical Society},
	VOLUME = {8},
	YEAR = {1957},
	PAGES = {535--536},
	MRCLASS = {46.2X},
	MRNUMBER = {87914},
	MRREVIEWER = {C.\ R.\ Putnam},
	DOI = {10.2307/2033511}
}

@article{Shi56,
	author = {Shirokov, F. V.},
	title = {Proof of a conjecture of Kaplansky},
	journal = {Uspekhi Matematicheskikh Nauk},
	year = {1956},
	volume = {11},
	number = {4(70)},
	pages = {167--168}
}

@book{MacLane,
	author = {{Mac\,Lane}, S.},
	edition = {Second Edition},
	publisher = {Springer},
	series = {Graduate Texts in Mathematics},
	title = {Categories for the Working Mathematician},
	volume = {5},
	year = {1998},
    isbn = {978-1-4757-4721-8},
	doi={10.1007/978-1-4757-4721-8}
	}

@article {MR4091410,
    AUTHOR = {Kim, S. and Lee, H.},
     TITLE = {Inequalities of the {W}asserstein mean with other matrix
              means},
   JOURNAL = {Annals of Functional Analysis},
  FJOURNAL = {Annals of Functional Analysis},
    VOLUME = {11},
      YEAR = {2020},
    NUMBER = {1},
     PAGES = {194--207},
      
   MRCLASS = {15B48 (15A45 47B65)},
  MRNUMBER = {4091410},
MRREVIEWER = {Tin-Yau Tam},
       DOI = {10.1007/s43034-019-00025-0}
}

@article{Bhatia_2018,
author = {Bhatia, R. and Jain, T. and Lim, Y.},
title = {Strong convexity of sandwiched entropies and related optimization problems},
journal = {Reviews in Mathematical Physics},
volume = {30},
number = {09},
pages = {1850014},
year = {2018},
doi = {10.1142/S0129055X18500149}
}

@article {MR3958140,
    AUTHOR = {Bhatia, R. and Jain, T. and Lim, Y.},
     TITLE = {Inequalities for the {W}asserstein mean of positive definite
              matrices},
   JOURNAL = {Linear Algebra and its Applications},
  FJOURNAL = {Linear Algebra and its Applications},
    VOLUME = {576},
    number = {6},
      YEAR = {2019},
     PAGES = {108--123},
      ISSN = {0024-3795},
   MRCLASS = {15A42 (15A18 47A30 47A64)},
  MRNUMBER = {3958140},
MRREVIEWER = {Mostafa Hayajneh},
       DOI = {10.1016/j.laa.2018.03.017}
}

@article {MR3992484,
    AUTHOR = {Bhatia, R. and Jain, T. and Lim, Y.},
     TITLE = {On the {B}ures-{W}asserstein distance between positive
              definite matrices},
   JOURNAL = {Expositiones Mathematicae},
  FJOURNAL = {Expositiones Mathematicae},
    VOLUME = {37},
      YEAR = {2019},
    NUMBER = {2},
     PAGES = {165--191},
      
   MRCLASS = {47A63 (47A64 53C20 53C22 60B05 81P45)},
  MRNUMBER = {3992484},
MRREVIEWER = {Rute Lemos},
       DOI = {10.1016/j.exmath.2018.01.002}
}

@article {MR1125815,
    AUTHOR = {Hofmann, K. H. and Strambach, K.},
     TITLE = {Topological and analytic loops},
 journal = {Quasigroups and loops: theory and applications},
    SERIES = {Sigma Series in Pure Mathematics},
    VOLUME = {8},
     PAGES = {205--262},
 PUBLISHER = {Heldermann, Berlin},
      YEAR = {1990},
   MRCLASS = {22A30 (20N05)},
  MRNUMBER = {1125815},
}

@article {MR0262413,
    AUTHOR = {Akivis, M. A.},
     TITLE = {The canonical expansions of the equations of a local analytic quasigroup},
   JOURNAL = {Doklady Akademii Nauk SSSR},
  FJOURNAL = {Doklady Akademii Nauk SSSR},
    VOLUME = {188},
      YEAR = {1969},
      number = {5},
     PAGES = {967--970},
    
   MRCLASS = {22.05},
  MRNUMBER = {0262413},
MRREVIEWER = {R. M. Daci\'{c}},
}

@book {MR0093552,
    AUTHOR = {Bruck, R. H.},
     TITLE = {A Survey of Binary Systems},
    SERIES = {Ergebnisse der Mathematik und ihrer Grenzgebiete, Neue Folge,
              Heft 20},
               PUBLISHER = {Springer Verlag, Berlin-G\"{o}ttingen-Heidelberg},
              YEAR = {1958},
             PAGES = {viii+185}
}

@book {MR1741419,
    AUTHOR = {Schaefer, H. H. and Wolff, M. P.},
     TITLE = {Topological Vector Spaces},
    SERIES = {Graduate Texts in Mathematics},
    VOLUME = {3},
   EDITION = {Second Edition},
 PUBLISHER = {Springer-Verlag, New York},
      YEAR = {1999},
     PAGES = {xii+346},
      ISBN = {0-387-98726-6},
       DOI = {10.1007/978-1-4612-1468-7}
}

@article{AndersonTrapp,
author = {Anderson, W. N.  and Trapp, G. E.},
title = {Operator means and electrical networks},
journal = {Proceedings 1980 IEEE International Symposium on Circuits and Systems},
year = {1980},
volume = {1},
pages = {523-527}
}

@book {MR1899153,
    AUTHOR = {Kiechle, H.},
     TITLE = {Theory of {$K$}-loops},
    SERIES = {Lecture Notes in Mathematics},
    VOLUME = {1778},
 PUBLISHER = {Springer-Verlag, Berlin},
      YEAR = {2002},
     PAGES = {x+186},
      ISBN = {3-540-43262-0},
       DOI = {10.1007/b83276}
}

@article {MR1780879,
    AUTHOR = {Sabinin, L. V.},
     TITLE = {Smooth quasigroups and loops: forty-five years of incredible
              growth},
   JOURNAL = {Commentationes Mathematicae Universitatis Carolinae},
  FJOURNAL = {Commentationes Mathematicae Universitatis Carolinae},
    VOLUME = {41},
      YEAR = {2000},
    NUMBER = {2},
     PAGES = {377--400},
   MRCLASS = {20N05 (53A60 53C22)},
  MRNUMBER = {1780879},
MRREVIEWER = {Huberta Lausch}}

\end{document}